\documentclass[preprint,12pt]{article}

\usepackage{amsmath}
\usepackage{amsthm}
\usepackage{latexsym}
	
	\DeclareMathOperator{\reg}{reg}
	\DeclareMathOperator{\height}{ht}

\usepackage[utf8]{inputenc}
\usepackage[english]{babel}
\usepackage{dsfont}
\usepackage{enumitem}

\usepackage{titling}
\predate{}
\postdate{}

\usepackage{tikz}
\usepackage{graphicx}
\graphicspath{ {./Diagrams/} }
\usepackage{multirow}

\usepackage{thmtools}
\usepackage{thm-restate}
\usepackage{hyperref}
\usepackage{cleveref}

	\theoremstyle{plain}
	\newtheorem{thm}{Theorem}[section]
	\newtheorem{prop}[thm]{Proposition}
	\newtheorem{cor}[thm]{Corollary}
	\newtheorem{lem}[thm]{Lemma}
	
	\theoremstyle{definition}
	\newtheorem{defn}[thm]{Definition}
	
	\theoremstyle{remark}

	\newtheorem{ex}[thm]{Example}

\usepackage{lineno}

\begin{document}
	
	\title{Dimensions of Betti Cones on Edge Ideals}
	\author{David Carey}
	\date{}
	\maketitle
	
	\begin{abstract}
		Boij-S\"{o}derberg Theory views the Betti diagrams of graded modules over polynomial rings as vectors in a $\mathds{Q}$-vector space, and studies the cone that these vectors generate (called a `Betti Cone'). The objects of study in this paper are the Betti cones generated by edge ideals. This paper presents and proves a formula for the dimensions of these cones, and for the subcones generated by edge ideals of specific heights.
	\end{abstract}
	
\section{Introduction}\label{sec: introduction}
	Throughout this paper, we fix a field $k$ (of arbitrary characteristic), and two non-negative integers $n$ and $h$ with $h<n$. We work inside the ring $R=k[x_1,...,x_n]$. We also let $[n]$ denote the set $\{1,...,n\}$.
	
	The central insight of Boij-S\"{o}derberg theory is that all Betti diagrams $\beta(M)$ of graded $R$-modules $M$ can be seen as vectors lying in the infinite-dimensional rational vector space $V_n=\bigoplus_{d\in \mathds{Z}} \mathds{Q}^{n+1}$, and that looking at the convex cone generated by these diagrams (called a `Betti Cone') provides a useful framework in which to study the diagrams themselves.
	
	This technique has proven to be very instructive in the study of Betti diagrams of arbitrary $R$-modules. We hope that it may prove similarly enlightening when applied more narrowly to the cones generated by certain classes of modules, such as edge ideals. The first step towards understanding these cones is to work out their dimensions.
	
	Edge ideals are squarefree monomial ideals in $R$ corresponding to graphs. They are defined as follows.
	\begin{defn}\label{def: edge ideal}
		Let $G$ be a graph with vertex set $[n]$. The edge ideal $G$, denoted $I(G)$, is the ideal in $R$ generated by $\{x_ix_j|\{i,j\}\in E(G)\}$.
	\end{defn}

	In this paper we restrict our attention to the Betti cone generated by all diagrams $\beta(I(G))$ for graphs $G$ on vertex set $[n]$, and the subcone of this cone generated by those graphs whose edge ideals have height $h$. We denote these cones by $\mathcal{C}_n$ and $\mathcal{C}_n^h$ respectively.
	
	We present and prove formulae for the dimensions of these cones, in terms of $n$ and $h$.
	In doing so, we exhibit, for each $n$, the minimal subspace $W_n$ of $V_n$ containing $\mathcal{C}_n$; and, for each $h$, we construct the minimal subspace $W_n^h$ of $W_n$ containing $\mathcal{C}_n^h$, as an intersection of defining hyperplanes in $W_n$.
	
	Suppose $\beta$ is a Betti diagram of an $R$-module. As is standard, we draw $\beta$ as a matrix $(a_{ij})$ with $a_{ij}=\beta_{i,i+j}$.
	
	\begin{equation*}\label{first-matrix}
		\begin{bmatrix}
			\vdots& \vdots& &\vdots \\
			\beta_{0,0} & \beta_{1,1} & \dots & \beta_{n,n} \\
			\beta_{0,1} & \beta_{1,2} & \dots & \beta_{n,n+1}\\
			\vdots & \vdots & & \vdots\\ 
		\end{bmatrix}
	\end{equation*}
	This matrix is infinite, but only finitely many of the entries are non-zero. For the diagrams in our cones $\mathcal{C}_n$ and $\mathcal{C}_n^h$, there are a lot of explicit restrictions on the positions of these non-zero values, which can help us to find the cones' dimensions.
	
	Ultimately, we work towards proving the following results.
	
	\begin{restatable}{thm}{dimCn}\label{dimCn}
		Let $\mathcal{C}_n$ be the cone generated by diagrams $\beta(I(G))$ for graphs $G$ on vertex set $[n]$. We have
		\begin{align*}
			\dim \mathcal{C}_n = \begin{cases}
				r^2 &\text{ if } n=2r \\
				r^2 + r &\text{ if } n=2r+1\, .\\
			\end{cases}
		\end{align*}
	\end{restatable}
	
	\begin{restatable}{thm}{dimCnh}\label{dimCnh}
		Let $\mathcal{C}_n^h$ be the cone generated by diagrams $\beta(I(G))$ for graphs $G$ on vertex set $[n]$ whose edge ideals have height $h$. We have
		\begin{align*}
			\dim \mathcal{C}_n^h = h(n-h-1) + 1\, .
		\end{align*}
	\end{restatable}
	
	Our proofs for both of these results proceed in roughly the same way: first, we bound the dimension from above by finding finite-dimensional subspaces of $V_n$ containing the cone; then we bound it from below by exhibiting an appropriately sized linearly independent set of vectors lying in the cone.
	
	\section{Preliminary Results}\label{sec: prelim-results}
	In this section, we present a number of important results on Betti diagrams of graphs. Throughout, we fix a graph $G$ and a simplicial complex $\Delta$, both on vertex set $[n]$. We set $\beta=\beta(I(G))$. We often write this as $\beta(G)$ for brevity.
	
	Because $I(G)$ is a squarefree monomial ideal, it is also the Stanley-Reisner ideal of some simplicial complex (see \cite{CCA}, Theorem 1.7). For a definition and treatment of Stanley-Reisner ideals and rings, we refer the reader to Definition 5.1.2 in \cite{C-M}, and the subsequent discussion. Following the notation of this text, we use $I_\Delta$ for the Stanley-Reisner ideal of a simplicial complex $\Delta$, and $k[\Delta]$ for the corresponding Stanley-Reisner ring $R/I_\Delta$.
	
	Specifically, $I(G)$ is the Stanley-Reisner ideal of the independence complex $\Delta(G)$ of $G$, whose faces are the sets of vertices of $G$ such that no two vertices share an edge in $G$ (see \cite{SFMon}, Lemma 2.15). The complex $\Delta(G)$ can also be thought of as the clique complex of the complement $G^c$ of $G$, whose faces are the cliques in $G^c$. This way of framing $I(G)$ is particularly useful, because it means $\beta$ can be thought of as $\beta(I_{\Delta(G)})$.
	
	\subsection{Key Tools}\label{sec: key tools}
	There are a number of powerful results that can help in classifying the diagrams that generate these cones. The two we use most in this paper are listed below.
	
	The first is Hochster's Formula (see \cite{MonIdeals}, Theorem 8.1.1), which allows us to compute the Betti numbers of Stanley-Reisner ideals from combinatorial data, in terms of the homology groups of their corresponding simplicial complexes. The second are the Herzog-K\"{u}hl equations (see \cite{Floy}, Section 1.3), which are linear dependency relations that hold between the Betti numbers of any graded module of a specified codimension.
	
	\begin{prop}[Hochster's Formula]\label{hoc}
		Let $\Delta$ be a simplicial complex on vertex set $[n]$. For any integers $i$ and $d$, we have
		\begin{equation*}
			\beta_{i,d}(I_\Delta) =\sum_{|U|=d} \dim_k \widetilde{H}_{d-i-2}(\Delta_U)\, .
		\end{equation*}
	\end{prop}
	
	Here, $\Delta_U$ denotes the induced subcomplex of $\Delta$ consisting of all faces contained in $U$, and $\widetilde{H}_j(\Delta_U)$ is the $j^\text{th}$ reduced homology group of $\Delta_U$ with coefficients in $k$.
	
	\begin{lem}[Herzog-K\"{u}hl Equations]\label{HK}
		Let $M$ be a graded $R$-module of codimension $c$, and let $\beta=\beta(M)$. For any integers $i$ and $d$, we have
		\begin{equation*}
			\sum_{i,d}(-1)^i d^j \beta_{i,d} = 0 \text{ for } j=0,...,c-1\, .
		\end{equation*}
	\end{lem}
	
	For convenience, for all $j=0,...,n$, we use the notation $HK_j(\beta)$ to represent the expression $\sum_{i,d}(-1)^i d^j \beta_{i,d}$. In this notation, the Herzog-K\"{u}hl equations for modules of codimension $c$ can be rephrased as the statement $HK_j(\beta)=0$ for each $j=0,...,c-1$.
	
	Recall that if
	
	\begin{equation*}\label{res-IDelta}
		0\rightarrow F_p \rightarrow \dotsb \rightarrow F_1 \rightarrow F_0
	\end{equation*}
	is a minimal graded free resolution of $I_\Delta$, then
	
	\begin{equation*}\label{res-kDelta}
		0\rightarrow F_p \rightarrow \dotsb \rightarrow F_1 \rightarrow F_0 \rightarrow R
	\end{equation*}
	is a minimal graded free resolution of $k[\Delta]$.
	
	This means we have $\beta_{0,0}(k[\Delta])=1$, and also $\beta_{i,d}(k[\Delta])=\beta_{i-1,d}(I_\Delta)$ for any integers $i$ and $d$ with $i>0$. Hence, for each $j>0$, we have
	\begin{align*}
		HK_j(\beta(k[\Delta])) &= \sum_{i,d}(-1)^i d^j \beta_{i,d}(k[\Delta])\\
		&= \sum_{i,d}(-1)^i d^j \beta_{i-1,d}(I_\Delta)\\
		&= - HK_j(\beta(I_\Delta))\, .
	\end{align*}

	The ideal $I_\Delta$ has height $h$ if and only if $k[\Delta]$ has codimension $h$. So if $I_\Delta$ has height $h$, then, by the Herzog-K\"{u}hl equations, we have $HK_j(\beta(I_\Delta))=0$ for each $j=1,...,h-1$.
	
	\subsection{Betti Diagrams of Specific Families of Graphs}\label{sec: Betti diagrams}
	In this section we present the Betti numbers of some specific families of graphs, along with an important proposition (Proposition \ref{beta-L}) that helps us construct Betti diagrams of slightly more complicated graphs.
	
	In what follows, we let $K_m$ denote the complete graph on $m$ vertices, $E_m$ denote the empty graph on $m$ vertices, and $C_m$ denote the cyclic graph on $m$ vertices. We also write $L$ for the graph consisting of a single edge, and $G+L$ for the disjoint union of $G$ and $L$.
	
	\begin{prop}\label{beta-iso}
		Let $G^*$ denote the graph obtained by removing all isolated vertices from $G$. The diagram $\beta(G)$ is equal to $\beta(G^*)$. In particular, we have $\beta(E_m)=0$ for any non-negative integer $m$.
	\end{prop}
	\begin{proof}
		The first part of this is immediate from the fact that $I(G)=I(G^*)$. The second part follows immediately from the fact that $E_m^*=\emptyset$.
	\end{proof}
	
	\begin{prop}\label{beta-Kn}
		Let $m$ be a non-negative integer. For any integers $i$ and $d$, we have
		\begin{equation*}
			\beta_{i,d}(K_m)=
			\begin{cases}
				(i+1){m\choose i+2} & \text{if } i=0,...,m-2 \text{ and } d=i+2\\
				0 & \text{otherwise.}
			\end{cases}
		\end{equation*}
		In particular, the diagram $\beta(K_m)$ has the following shape.
		\begin{equation*}
			\begin{bmatrix}
				\beta_{0,2} & \dots & \beta_{m-2,m}\\
			\end{bmatrix}
		\end{equation*}
	\end{prop}
	\begin{proof}
		This is from Theorem 5.1.1 in \cite{Jacques}, and the fact that for any integers $i$ and $d$ with $i>0$, we have $\beta_{i,d}(k[\Delta])=\beta_{i-1,d}(I_\Delta)$.
	\end{proof}
	
	\begin{prop}\label{beta-Cn}
		Let $m$ be a non-negative integer, and suppose the complement $G^c$ of $G$ is equal to $C_m$. For any integers $i$ and $d$, we have
		\begin{equation*}
			\beta_{i,d}(G)=
			\begin{cases}
				\frac{m(i+1)}{m-i-2}{m-2\choose i+2}& \text{if } i=0,...,m-4 \text{ and } d=i+2\\
				1 & \text{if } (i,d) = (m-3,m) \\
				0 & \text{otherwise.}
			\end{cases}
		\end{equation*}
		In particular, the diagram $\beta((C_m)^c)$ has the following shape.
		\begin{equation*}
			\begin{bmatrix}
				\beta_{0,2} & \dots & \beta_{m-4,m-2} & \\
				& & &\beta_{m-3,m}\\ 
			\end{bmatrix}\, .
		\end{equation*}
	\end{prop}
	\begin{proof}
		See Theorem 2.3.3 in \cite{Ramos}.
	\end{proof}
	
	We end with a particularly important proposition that allows us to find the Betti diagram of the suspension $S\Delta$ of $\Delta$ from the Betti diagram of $\Delta$.
	
	Note that the independence complex of $G+L$ is equal to the suspension of the independence complex of $G$, so this proposition also allows us to find the Betti Diagram of $G+L$ from the Betti diagram of $G$.
	
	\begin{prop}\label{beta-L}
		For any integers $i$ and $d$, we have $\beta_{i,d}(S\Delta)=\beta_{i,d}(\Delta)+\beta_{i-1,d-2}(\Delta)$. In particular, we have $\beta_{i,d}(G+L)=\beta_{i,d}(G)+\beta_{i-1,d-2}(G)$.
	\end{prop}
	
	\begin{proof}
		This follows from Hochster's Formula, using the fact that for any topological space $X$ and any $i\geq0$, we have $\widetilde{H}_{i}(SX)\cong\widetilde{H}_{i-1}(X)$.
	\end{proof} 
	
\section{Indexing Sets and Linear Independence}\label{sec: index}
	In the following sections, we establish formulae for the dimensions of the cones $\mathcal{C}_n$ and $\mathcal{C}_n^h$. Our proofs proceed by showing that the formulae given are both upper and lower bounds for the dimensions of the cones.
	
	To find a lower bound $l$ for the dimension of a convex cone $\mathcal{C}$, it suffices to find a linearly independent set of $l$ vectors lying in $\mathcal{C}$, as this shows that the smallest vector space containing $\mathcal{C}$ must have dimension at least $l$. In this section, for ease of explanation, we present terminology for a simple condition that ensures linear independence.
	
	Suppose $\mathcal{C}$ lives inside the rational vector space $V=\bigoplus_{i\in I} \mathds{Q}$ for some finite indexing set $I$. For a vector $v$ in $V$ and an index $i\in I$, let $v_i$ denote the $i^\text{th}$ coordinate of $v$. Also suppose we have a strict total ordering $\prec$ on $I$.
	
	\begin{defn}\label{def: initial}
		Let $v\in V$ and $i\in I$. We say $v$ is \textit{$i$-initial with respect to $\prec$} (which we often write as \textit{$i_\prec$-initial}, or just \textit{$i$-initial} when doing so does not result in ambiguity) if
		\begin{enumerate}
			\item The component $v_i$ is non-zero;
			\item For every $j\in I$ such that $i\prec j$, the component $v_j$ equals zero.
		\end{enumerate}
	\end{defn}
	
	If $X=\{v^i\}_{i\in I}$ is a set of vectors lying in $\mathcal{C}$ such that for each $i\in I$, $v^i$ is $i_\prec$-initial, then $X$ must be linearly independent. So to find a linearly independent set of vectors in $\mathcal{C}$, it suffices to define an order $\prec$ on $I$, and find an $i_\prec$-initial vector lying in $\mathcal{C}$ for each $i$ in $I$.
	
\section{Dimension of $C_n$}
	In this section, we prove Theorem \ref{dimCn}.
	
	\subsection{Upper Bound}\label{sec: ub Cn}
	We start by bounding the dimension from above. Every Betti diagram lives in the infinite dimensional vector space $V_n = \bigoplus_{d\in \mathds{Z}} \mathds{Q}^{n+1}$. We work towards finding a finite indexing set $S_n\subset \{0,...,n\}\times\mathds{Z}$ such that for every $\beta\in \mathcal{C}_n$, and every integer $d$ and $i=0,...,d$ with $(i,d)\notin S_n$, we have $\beta_{i,d}=0$. This demonstrates that $\mathcal{C}_n$ actually lies inside the finite-dimensional vector space $W_n = \bigoplus_{(i,d)\in S_n}\mathds{Q}$, and hence we have $\dim \mathcal{C}_n\leq \dim W_n = |S_n|$.
	
	To find our indexing set $S_n$, we need to obtain some restrictions on the positions of the non-zero values of the diagrams in $\mathcal{C}_n$.
	
	\begin{prop}\label{inequalities}
		Consider $\beta \in \mathcal{C}_n$.  For
		any non-negative integer $i$ less than or equal to $n$ and any integer $d$
		satisfying $d < i+2$, $d > n$, or $d > 2i +2$, we have $\beta_{i,d} = 0$.
	\end{prop}
	
	The conditions $d< i+2$ and $d>n$ are simple corollaries of Hochster's Formula, and hold for all Stanley-Reisner ideals of degree 2 or higher, while the final condition $d>2i+2$ only holds for edge ideals. A proof for this final condition can be found in Lemma 2.2 in \cite{Katz}.
	
	These inequalities give us a much clearer picture of the shape of the Betti diagrams in $\mathcal{C}_n$. Specifically, if $n = 2r$ is even, then the diagrams $\beta\in \mathcal{C}_n$ look like this.
	\begin{equation}
		\begin{bmatrix}\label{matrix-C_2r}
			\beta_{0,2} & \beta_{1,3} & \beta_{2,4} & \dots & \dots & \beta_{n-3,n-1} & \beta_{n-2,n}\\
			& \beta_{1,4} & \beta_{2,5} & \dots & \dots & \beta_{n-3,n} &\\
			& & \ddots & & & &\\
			& & & \beta_{r-1,2r} & & & \\
		\end{bmatrix}
	\end{equation}
	If $n = 2r+1$ is odd, then they look like this.
	\begin{equation}
		\begin{bmatrix}\label{matrix-C_2r+1}
			\beta_{0,2} & \beta_{1,3} & \beta_{2,4} & \dots & \dots & \dots & \beta_{n-3,n-1} & \beta_{n-2,n}\\
			& \beta_{1,4} & \beta_{2,5} & \dots & \dots & \dots & \beta_{n-3,n} &\\
			& & \ddots & & & & &\\
			& & & \beta_{r-1,2r} & \beta_{r, 2r+1} & & & \\
		\end{bmatrix}
	\end{equation}

	 Thus, we may define our indexing set $S_n$ and subspace $W_n$ as follows.
	 \begin{defn}\label{def: Sn, Wn}
	 	For a fixed non-negative integer $n$, we define
	 	\begin{enumerate}
	 		\item $S_n := \{(i,d)\in \{0,...,n\}\times \mathds{Z}| i+2\leq d\leq \min\{2i+2,n\} \}$.
	 		\item $W_n := \bigoplus_{(i,d)\in S_n}\mathds{Q}$.
	 	\end{enumerate}
	 \end{defn}
 	
 	By Proposition \ref{inequalities}, the cone $\mathcal{C}_n$ must lie in $W_n$ as desired.
	 
	For ease of explanation, it is sometimes useful for us to refer to individual rows of $S_n$.
	 
	 \begin{defn}\label{rows}
	 	Let $(i,d)\in S_n$. We say $(i,d)$ is in row $\rho$ if we have $d-i-1=\rho$.
	 \end{defn}
 	
 	We can arrange the elements of $S_{2r}$ in rows as in equation (\ref{matrix-C_2r}).
 	\begin{equation}\label{matrix-S_2r}
 		\begin{matrix}
 			(0,2) & (1,3) & (2,4) & \dots & \dots & (n-3,n-1) & (n-2,n)\\
 			& (1,4) & (2,5) & \dots & \dots & (n-3,n) &\\
 			& & \ddots & & & &\\
 			& & & (r-1,2r) & & & \\
 		\end{matrix}
 	\end{equation}
 	Similarly, we can arrange the elements of $S_{2r+1}$ in rows as in equation (\ref{matrix-C_2r+1}).
 	\begin{equation}\label{matrix-S_2r+1}
 		\begin{matrix}
 			(0,2) & (1,3) & (2,4) & \dots & \dots & \dots & (n-3,n-1) & (n-2,n)\\
 			& (1,4) & (2,5) & \dots & \dots & \dots & (n-3,n) &\\
 			& & & \ddots & & & & &\\
 			& & & & (r-1,2r) & (r,2r+1) &  & & \\
 		\end{matrix}
 	\end{equation}
 	
 	We can see that row 1 of $S_n$ has $n-1$ elements, row 2 has $n-3$ elements, and so on. In general, for each $i=1,...,r$, row $i$ of $S_n$ has $n-2i+1$ elements.
 	
 	Hence we have
 	\begin{align*}
 		|S_{2r}|&= \sum_{i=1}^r (2r-2i+1)\\
 		&= \sum_{i=1}^r (2r) - \sum_{i=1}^{r}(2i-1)\\
 		&= 2r^2 - r^2\\
 		&= r^2
 	\end{align*}
 	
 	and
 	\begin{align*}
 		|S_{2r+1}|&=\sum_{i=1}^r (2r+1-2i+1)\\
 		&= \sum_{i=1}^r (2r+1) - \sum_{i=1}^{r}(2i-1)\\
 		&= (2r^2 + r) - r^2\\
 		&= r^2 + r \, .
 	\end{align*}
 
 	Therefore, the expressions in Theorem \ref{dimCn} are upper bounds for $\dim \mathcal{C}_n$.
	
	\subsection{Lower Bound}\label{sec: lb Cn}
	To complete our proof of Theorem \ref{dimCn}, it only remains to show that the space $W_n$ is the \textit{minimal} subspace of $V_n$ containing $\mathcal{C}_n$. We work towards finding a linearly independent set of Betti diagrams lying in $\mathcal{C}_n$, of the same size as $S_n$. This shows us that $\dim \mathcal{C}_n$ is at least as large as the cardinality of $S_n$, and hence we have $\dim \mathcal{C}_n = |S_n|$.
	
	In defining our linearly independent set of diagrams, we reuse the notation $K_m$ and $L$ introduced in section \ref{sec: Betti diagrams}.
	
	The following lemma is helpful, because it allows us to broaden our search from graphs with exactly $n$ vertices to graphs with at most $n$ vertices.
	
	\begin{lem}\label{Cm in Cn}
		For any positive integer $m$
		satisfying $m < n$, we have $\mathcal{C}_m \subset \mathcal{C}_n$.
	\end{lem}
	\begin{proof}
		If $G$ is a graph on $[m]$, then by Proposition \ref{beta-iso}, we can extend it to a graph on $[n]$ by adding some isolated vertices, without affecting its Betti diagram. This means the diagram $\beta(G)$ lies inside $\mathcal{C}_n$, and the result follows.
	\end{proof}
	
	To find these diagrams, we use the
	concept of initiality introduced in Definition \ref{def: initial}. Specifically, we want to define an ordering $\prec$ on $S_n$ and then find an $(i,d)_\prec$-initial diagram in $\mathcal{C}_n$ for each $(i,d)$ in $S_n$. For convenience, we also extend our terminology to say that a graph $G$ is $(i,d)_\prec$-initial if its Betti diagram $\beta(G)$ is $(i,d)_\prec$-initial.
	
	\begin{defn}\label{def: order-Cn}
		For any two pairs $(i,d)$ and $(i',d')$ in $S_n$ we write $(i,d)\prec (i',d')$ if $d-i< d'-i'$, or $d-i= d'-i'$and $i<i'$.
	\end{defn}

	In other words we say $(i,d)\prec(i',d')$ if $(i,d)$ lies in a lower numbered row, or if they both lie in the same row with $i<i'$.

	This ordering is particularly useful to us for the following reason.
	\begin{lem}\label{G+L initiality}
		Let $\prec$ be as in Definition \ref{def: order-Cn}, let $(i,d)\in S_n$, and suppose $G$ is an $(i-1,d-2)_\prec$-initial graph on $[n-2]$. The graph $G+L$ is $(i,d)_\prec$-initial.
	\end{lem}
	\begin{proof}
		By Proposition \ref{beta-L}, we have $\beta_{i,d}(G+L) = \beta_{i,d}(G) + \beta_{i-1,d-2}(G)$. By the $(i-1,d-2)_\prec$-initiality of $\beta(G)$, we have $\beta_{i-1,d-2}(G) \neq 0$, so $\beta_{i,d}(G+L)$ must be non-zero too.
		
		Now let $(i',d')\in S_n$ with $(i,d)\prec (i',d')$. Again, by Proposition \ref{beta-L} we have $\beta_{i',d'}(G+L) = \beta_{i',d'}(G) + \beta_{i'-1,d'-2}(G)$. We must have $(i-1,d-2)\prec (i'-1,d'-2)$, and also $(i'-1,d'-2)\prec_{h-1} (i',d')$ because they are in different rows. Hence, by the $(i-1,d-2)_\prec$-initiality of $\beta(G)$, both the terms $\beta_{i',d'}(G)$ and $\beta_{i'-1,d'-2}(G)$ are zero, and $\beta_{i',d'}(G+L)$ is zero too.
		
		This shows that $\beta(G+L)$ is $(i,d)_\prec$-initial as required.
	\end{proof}
	
	Before we explain the procedure for finding linearly independent diagrams in the general case, we present a specific example to illustrate the broad principles.
	
	\begin{ex}
		If $n=6$, then $S_n$ has size $3^2 = 9$, and it looks like the following.
		\begin{equation*}
			\begin{matrix}
				(0,2) & (1,3) & (2,4) & (3,5) & (4,6)\\
				& (1,4) & (2,5) & (3,6) &\\
				& & (2,6) & &
			\end{matrix}
		\end{equation*}
	
	So we want to find nine linearly independent diagrams in $\mathcal{C}_6$, one for each $(i,d)\in S_6$.
	
	The ordering $\prec$ on $S_6$ is $(0,2)\prec (1,3) \prec (2,4) \prec (3,5) \prec (4,6) \prec (1,4) \prec (2,5) \prec (3,6) \prec (2,6)$.
	
	By Proposition \ref{beta-Kn}, we see that the complete graph on $2$ vertices, $K_2$, is $(0,2)$-initial. Similarly, $K_3$ is $(1,3)$-initial, $K_4$ is $(2,4)$-initial, $K_5$ is $(3,5)$-initial and $K_6$ is $(4,6)$-initial.
	
	From the above, and Lemma \ref{G+L initiality}, we also find that $K_2+L$ is $(1,4)$-initial, $K_3+L$ is $(2,5)$-initial and $K_4+L$ is $(3,6)$-initial. Similarly, we can see that $K_2+2L$ is $(2,6)$-initial.
	
	So placing each graph in its corresponding position in $S_n$, we get the following.
	\begin{equation*}
		\begin{matrix}
			K_2 & K_3 & K_4 & K_5 & K_6\\
			& K_2+L & K_3+L & K_4+L &\\
			& & K_2+2L & &
		\end{matrix}
	\end{equation*}
	
	All of the graphs $K_2,K_3,K_4,K_6,K_2+L,K_3+L,K_4+L$ and $K_2+2L$ have $6$ vertices or less, so by Lemma \ref{Cm in Cn}, their diagrams all lie in $\mathcal{C}_6$ as required.\\
	\end{ex}
	
	In the above example, the graphs associated to the top row of $S_n$ were the complete graphs on $n$ or fewer vertices, and we found graphs for each subsequent row by adding disjoint edges to the graphs we had already found. We can generalise this process to arbitrary values of $n$, as follows.
	
	\begin{prop}\label{dimCn lower bound}
		Let $S_n$ be as in definition \ref{def: Sn, Wn}. We have $\dim \mathcal{C}_n \geq |S_n|$.
	\end{prop}
	\begin{proof}
		We show that there exists a set of graphs $\{G_{i,d}|(i,d)\in S_n\}$ such that for each $(i,d)$ in $S_n$, $G_{i,d}$ is $(i,d)_\prec$-initial and has $d$ vertices. This proves the result, because it means that for each $(i,d)$ in $S_n$ we have $d\leq n$, and hence the Betti diagram of these graphs all lie in $\mathcal{C}_n$ by Lemma \ref{Cm in Cn}.
		
		We proceed by induction on $n\geq 1$. The set $S_1$ is empty, so for the base case $n=1$ there is nothing to prove.
		
		For the inductive step, suppose that we have a set $\{G_{i,d}|(i,d)\in S_{n-1}\}$ where each $G_{i,d}$ is an $(i,d)$-initial graph on $d$ vertices. The set $S_{n-1}$ is a subset of $S_n$, so we can extend our set of graphs to a set $\{G_{i,d}|(i,d)\in S_n\}$ by adding graphs $G_{i,d}$ for the values of $(i,d)$ in $S_n-S_{n-1}$.
		
		By Proposition \ref{beta-Kn}, the complete graph $K_n$ is $(n-2,n)$-initial and has $n$ vertices, so we can set $G_{n-2,n}=K_n$.
		
		From the diagrams \ref{matrix-S_2r} and \ref{matrix-S_2r+1} at the end of Section \ref{sec: ub Cn}, we can see that for every other value of $(i,d)$ in $S_n-S_{n-1}$, $(i-1,d-2)$ is in $S_{n-1}$. Hence, we can define $G_{i,d}=G_{i-1,d-2}+L$. This graph has $(d-2)+2=d$ vertices and by Lemma \ref{G+L initiality}, we know it must be $(i,d)$-initial. This completes the proof.
	\end{proof}

\section{Dimension of $C_n^h$}
	In this section, we prove Theorem \ref{dimCnh}.
	
	\subsection{Upper Bound}
	Again, we start by showing that the formula given in Theorem \ref{dimCnh} is an upper bound for the dimension of $\mathcal{C}_n^h$, by finding a subset $S_n^h\subset S_n$ such that for every $\beta\in \mathcal{C}_n^h$, and every integer $d$ and $i=0,...,d$ with $(i,d)\notin S_n^h$, we have $\beta_{i,d}=0$. This demonstrates that $\mathcal{C}_n^h$ actually lies inside the subspace $(W_n^h)' = \bigoplus_{(i,d)\in S_n^h}\mathds{Q}$, and hence that we have $\dim \mathcal{C}_n^h\leq \dim (W_n^h)' = |S_n^h|$.
	
	As will become clear when we find our candidate for $S_n^h$, the dimension of $(W_n^h)'$ is slightly larger than $h(n-h-1)$. This is because every diagram $\beta$ in $\mathcal{C}_n^h$ must satisfy $HK_j(\beta)=0^j$ for all $j=0,...,h-1$ (by the argument at the end of Section \ref{sec: key tools}) and hence they all lie in the proper subspace $W_n^h = \{\beta\in (W_n^h)'| HK_j(\beta)=0 \text{ for } j=1,...,h-1\}$ of $(W_n^h)'$. We show that this latter space $W_n^h$ has the correct dimension.
	
	To find our indexing set $S_n^h$, we need to obtain some restrictions on the positions of the non-zero values of the diagrams in $\mathcal{C}_n^h$. To help us with this, we need two important lemmas.
	
	First, recall the following definitions.
	\begin{enumerate}
		\item A \textbf{vertex cover} for a graph $G$ is a subset $U\subseteq [n]$ such that for every edge $e$ in $G$, the set $e\cap U$ is non-empty.
		\item A \textbf{matching} in $G$ is a collection $C$ of pairwise disjoint edges in $G$. We call such a matching maximal if it is maximal with respect to inclusion.
		\item The \textbf{regularity} $\reg M$ of a graded $R$-module $M$ is (among other things) the maximum value of $d-i$ such that $\beta_{i,d}(M)$ is non-zero.
	\end{enumerate}
	
	Our two lemmas are below. Lemma \ref{reg} is Theorem 4.4 in \cite{SFMon}.
	\begin{lem}\label{ht}
		Let $G$ be a graph on vertex set $[n]$, and $\Delta$ be the independence complex of $G$. The following are equivalent.
		\begin{enumerate}
			\item The height of $I(G)$ is equal to $h$.
			\item $G$ has a minimally sized vertex cover of size $h$.
			\item The minimum number of vertices needed to be removed from $G^c$ to obtain a complete graph is $h$.
			\item $h = n- \dim \Delta - 1$.
		\end{enumerate}
	\end{lem}
	\begin{proof}
		For the equivalence of (1) and (2), see Corollary 7.2.4 in \cite{Vil}. The equivalence of (2) and (3) follows from the fact that $U$ is a vertex cover for $G$ if and only if $G-U$ has no edges, which means it is the complement of a complete graph. The equivalence of (1) and (4) follows directly from Theorem 5.1.4 in \cite{C-M}, and the fact that the Krull dimension of $k[\Delta]$ is equal to $n-\height I_\Delta$.
	\end{proof}
	
	\begin{lem}\label{reg}
		Let $G$ be a graph on vertex set $[n]$, and let $\alpha$ be the minimum size of a maximal matching in $G$. We have $\reg I(G)\leq \alpha + 1$.
	\end{lem}
	
	Using these two lemmas, we can prove the following proposition.
	\begin{prop}\label{Snh}
		Consider $\beta\in \mathcal{C}_n^h$. For every $(i,d)\in S_n$ satisfying $d - i > \min\{h,n-h\}+1$, we have $\beta_{i,d}=0$.
	\end{prop}
	
	\begin{proof}
		Let $G$ be a graph on $n$ vertices of height $h$. We need to show that $\reg I(G)$ is less than or equal to both $h+1$ and $n-h+1$.
		
		To show that we have $\reg I(G)\leq h+1$, first note that, by Lemma \ref{ht}, there is a minimal vertex cover $\{x_{i_1},...,x_{i_h}\}$ for $G$. This means that every edge in $E(G)$ must contain at least one of $x_{i_1},...x_{i_h}$. Hence, no matching in $G$ can consist of more than $h$ edges, so the minimal size of a maximal matching in $G$ must be less than or equal to $h$. By Lemma \ref{reg}, we have $\reg I(G)\leq h+1$.
		
		To show that $\reg I(G)\leq n-h+1$, we appeal to Hochster's Formula. Recall that this was
		\begin{equation*}
			\beta_{i,d}(I_\Delta) =\sum_{|U|=d} \dim_k \widetilde{H}_{d-i-2}(\Delta_U)\, .
		\end{equation*}
		
		Let $\beta=\beta(I(G))$, $\Delta$ be the indepndence complex of $G$, and $U$ be a subset of $[n]$.
	
		Suppose $\Delta$ has dimension $D$. This means that $\Delta$ has no faces of dimension higher than $D$, and in particular, $\Delta_U$ has no faces of dimension higher than $D$ either. Hence for any $j> D$, we have $\widetilde{H}_j(\Delta_U) = 0$.
		
		Let $(i,d)\in S_n$ be such that $\beta_{i,d}\neq 0$. By Hochster's Formula, and the argument above, we must have that $d-i-2\leq D$.
		
		By Lemma \ref{ht}, we know that $D=n-h-1$. Thus we conclude that $d-i\leq n-h+1$, as required.
	\end{proof}

	These inequalities give us a much clearer picture of what the diagrams in $\mathcal{C}_n^h$ look like. Setting $m=\min\{h,n-h\}$, we get that the diagrams $\beta\in \mathcal{C}_n^h$ look like this.
	\begin{equation}
		\begin{bmatrix}\label{matrix-C_n^h}
			\beta_{0,2} & \dots & \dots & \dots & \dots & \dots & \beta_{n-2,n}\\
			&\ddots & & & & &\\
			& & \beta_{m-1,2m} &\dots &\beta_{n-m-1,n} & \\
		\end{bmatrix}
	\end{equation}

	Thus, we may define our indexing set $S_n^h$ and our subspaces $(W_n^h)'$ and $W_n^h$ as follows.
	
	Thus, we may define our indexing set $S_n$ and subspace $W_n$ as follows.
	\begin{defn}\label{def: Snh, Wnh}
		For fixed non-negative integers $h<n$, we define
		\begin{enumerate}
			\item $S_n^h := \{(i,d)\in S_n| d-i\leq \min\{h,n-h\}+1\}$.
			\item $(W_n^h)' := \bigoplus_{(i,d)\in S_n^h}\mathds{Q}$.
			\item $W_n^h := \{\beta\in (W_n^h)'| HK_j(\beta)=0 \text{ for } j=1,...,h-1\}$.
		\end{enumerate}
	\end{defn}
	
	By Proposition \ref{Snh} and Lemma \ref{HK}, the cone $\mathcal{C}_n^h$ must lie in $W_n^h$ as desired.
	
	To prove that the fomula given in Theorem \ref{dimCnh} is an upper bound for $\dim \mathcal{C}_n^h$, we only need to show that $\dim W_n^h=h(n-h-1)$, which we do in two parts.
	
	\begin{prop}\label{|S_n^h|}
		Let $S_n^h$ be as in definition \ref{def: Snh, Wnh}. We have $|S_n^h|=h(n-h)$.
	\end{prop}
	\begin{proof}
		Let $m=\min\{h,n-h\}$. The set $S_n^h$ looks like the following.
		\begin{equation*}\label{S_n^h}
			\begin{matrix}
				(0,2) & \dots & \dots & \dots & \dots & \dots & (n-1,n)\\
				&\ddots & & & & &\\
				& & (m-1,2m) &\dots &(n-m-1,n) & \\
			\end{matrix}
		\end{equation*}
		
		As noted in Section \ref{sec: ub Cn}, for all $i=1,...,m$, row $i$ of $S_n$ has $n-2i+1$ elements.
		
		Hence we have
		\begin{align*}
			|S_n^h|&= \sum_{i=1}^m (n-2i+1)\\
			&= \sum_{i=1}^m n - \sum_{i=1}^m(2i-1)\\
			&= nm - m^2\\
			&= m(n-m)\\
			&=h(n-h)\, .
		\end{align*}
	\end{proof}

	\begin{prop}\label{dimCnh-upper}
		Let $W_n^h$ be as in definition \ref{def: Snh, Wnh}. We have $\dim W_n^h = h(n-h-1) + 1$.
	\end{prop}
	\begin{proof}
		Recall that $W_n^h$ is defined as the subspace of $(W_n^h)'$ consisting of all the diagrams $\beta$ for which $HK_1(\beta)=...=HK_{h-1}(\beta)=0$.
		
		Hence it is enough to show that the linear forms $HK_1(\beta),...,HK_{h-1}(\beta)$ are linearly independent, because if this is the case then we have
		\begin{align*}
			\dim W_n^h &\leq \dim (W_n^h)' - (h-1)\\
			&= |S_n^h|-h+1\\
			&= h(n-h)-h+1\\
			&= h(n-h-1)+1\, .
		\end{align*}
		
		To this end we define, for each $d=1,...,n$,
		\begin{align*}
			t_d = \sum_{i} (-1)^i \beta_{i,d}\, .
		\end{align*}
		
		The relations $HK_1(\beta)= ... = HK_{h-1}(\beta)=0$ may be expressed as
		\begin{equation*}
			(t_1,...,t_n)
			\begin{pmatrix}
				1 & \dots & 1\\
				2 & \dots & 2^{h-1}\\
				\vdots & \vdots & \vdots\\
				n & \dots & n^{h-1}
			\end{pmatrix}
			=0\, .
		\end{equation*}
		
		The matrix of coefficients given above is a Vandermonde matrix with distinct rows, which means in particular that all of its columns are linearly independent.
		
		Thus the Herzog-K\"{u}hl equations are linearly independent in the indeterminates $t_1,..,t_n$, and hence linearly independent in the indeterminates $\beta_{i,d}$ for $(i,d)\in S_n^h$. This completes the proof.
	\end{proof}
	
	\subsection{Lower Bound}
	To complete our proof of Theorem \ref{dimCnh}, it only remains to show that the space $W_n^h$ is the \textit{minimal} subspace of $W_n$ containing $\mathcal{C}_n^h$.
	
	Just as in Section \ref{sec: lb Cn}, we do this by finding a linearly independent set of Betti diagrams lying in $\mathcal{C}_n^h$ of size $h(n-h-1)+1$. This shows that $\dim \mathcal{C}_n$ is at least $h(n-h-1) + 1$, and hence we have $\dim \mathcal{C}_n = h(n-h-1) + 1$.
	
	In defining our linearly independent set of diagrams, we reuse the notation $K_m$, $E_m$, $C_m$ and $L$ introduced in section \ref{sec: Betti diagrams}.
	
	The following lemma is an analogue of Lemma \ref{Cm in Cn}.
	
	\begin{lem}\label{Cmh in Cnh}
		For any positive integer $m$
		satisfying $m < n$, we have $\mathcal{C}_m^h \subset \mathcal{C}_n^h$.
	\end{lem}
	\begin{proof}
		Let $G$ be a graph on $[m]$ of height $h$, and extend it to a graph on $[n]$ by adding some isolated vertices. By Proposition \ref{beta-iso}, this extension does not affect the graph's Betti diagram, and by Lemma \ref{ht}, it does not affect the graph's height. This means the diagram $\beta(G)$ lies inside $\mathcal{C}_n^h$, and the result follows.
	\end{proof}
	
	To help us in finding our linearly independent diagrams in $\mathcal{C}_n^h$, we use the following proposition and corollaries. For all three, we use the notation $\beta^c(G)$ for the diagram $\beta(G^c)$ to avoid over-using parentheses in longer expressions.
	\begin{prop}\label{beta-adding-Em}
		Let $l$ and $m$ be non-negative integers with $l<m$. Suppose $G$ is a graph on $[l]$, and let  $\beta^c = \beta^c(G)=\beta(G^c)$. Define $\widetilde{\beta}^c=\beta^c(G+E_{m-l})$.
		\begin{enumerate}
			\item for every $0\leq i \leq m -2$, we have $\widetilde{\beta}^c_{i,i+2}\neq 0$.
			\item for every $(i,d)$ in $S_m$ with $d-i\geq 3$, we have $\widetilde{\beta}^c_{i,d}=\sum_{j=0}^{i} \beta_{j,j+d-i}$.
		\end{enumerate}
	\end{prop}
	\begin{proof}
		Let $\Delta$ be the complex of cliques of $G$ and $\widetilde{\Delta}$ be the complex of cliques of $G+E_{m-l}$. The complex $\widetilde{\Delta}$ can be obtained from $\Delta$ by adding $m-l$ isolated vertices. We assume $\widetilde{\Delta}$ has vertex set $[m]$ and label these additional vertices $l+1,...,m$.
		
		For part (1), we note that for any subset $U\subseteq [m]$ that contains both a vertex in $[l]$ and a vertex in $[m]-[l]$, the complex $\widetilde{\Delta}_U$ must be disconnected. Or in other words, we have $\widetilde{H}_0(\widetilde{\Delta})\neq 0$. By Hochster's Formula, we get that $\widetilde{\beta}^c_{0,2},...,\widetilde{\beta}^c_{m-2,m}\neq 0$.
		
		For part (2), suppose $(i,d)\in S_m$ with $d-i\geq 3$. The addition of isolated vertices to $\Delta$ has no affect on homologies of degree greater than zero. This means that for a subset $U\subseteq [m]$, the only part of $U$ that contributes to the $(d-i-2)^\text{nd}$ homology of $\widetilde{\Delta}_U$ is $U\cap [l]$, and so we have $\widetilde{H}_{d-i-2}(\widetilde{\Delta}_U) = \widetilde{H}_{d-i-2}(\Delta_{U\cap[l]})$. Thus, by Hochster's Formula, we get
		
		\begin{align*}
			\widetilde{\beta}^c_{i,d} &=\sum_{\substack{U\subseteq[m]\\ |U|=d}} \dim_k \widetilde{H}_{d-i-2}(\widetilde{\Delta}_U)\\
			&=\sum_{r=0}^{d} \text{      } \sum_{\substack{U\subseteq[m]\\ |U|=d \\ |U\cap [l]|=r}} \dim_k \widetilde{H}_{d-i-2}(\widetilde{\Delta}_U)\\
			&=\sum_{r=0}^d \text{      } \sum_{\substack{U\subseteq[l]\\ |U|=r}} \dim_k \widetilde{H}_{d-i-2}(\Delta_U)\\
			&=\sum_{r=0}^d \text{      } \sum_{\substack{U\subseteq[l]\\ |U|=r}} \dim_k \widetilde{H}_{r-(r+i-d)-2}(\Delta_U)\\
			&=\sum_{r=0}^d \beta_{r+i-d,r}\\
			&=\sum_{j=0}^i \beta_{j,j+d-i} \, .
		\end{align*}
	\end{proof}

	\begin{cor}\label{cor-Kl}
		Let $l$ and $m$ be non-negative integers with $l<m$. We have
		\begin{equation*}
			\beta^c(K_l + E_{m-l}) =
			\begin{bmatrix}
				\beta_{0,2} & \dots & \beta_{m-2,m}
			\end{bmatrix}\, .
		\end{equation*}
	\end{cor}
	\begin{proof}
		The diagram $\beta^c(K_l)$ is equal to $\beta(E_l)$, so this follows directly from Proposition \ref{beta-iso} and Proposition \ref{beta-adding-Em}.
	\end{proof}
	
	\begin{cor}\label{cor-Cl}
		Let $l$ and $m$ be non-negative integers with $l<m$. We have
		\begin{equation*}
			 \beta^c(C_l + E_{m-l}) =
			\begin{bmatrix}
				\beta_{0,2} & \dots & \beta_{l-3,l-1} & \dots & \beta_{m-2,m}\\
				& & \beta_{l-3,l} & \dots & \beta_{m-3,m} 
			\end{bmatrix}\, .
		\end{equation*}
	\end{cor}
	\begin{proof}
		This follows directly from Proposition \ref{beta-Cn} and Proposition \ref{beta-adding-Em}.
	\end{proof}
	
	Just as in Section \ref{sec: lb Cn}, our key tool for finding our linearly independent diagrams is the concept of initiality presented in Section \ref{sec: index}. 
	
	However, there are two ways in which we need to refine our approach. The first is that we need to use a different ordering from the one given in Definition \ref{def: order-Cn}, and the second is that we only need to find $h(n-h-h) +1$ diagrams, whereas $|S_n^h|=h(n-h)$ so there are $(h-1)$ values of $(i,d)$ in $S_n^h$ for which we do not need to find $(i,d)$-initial graphs.
	
	To motivate our new approach, we consider the following example.
	
	\begin{ex}\label{ex-Cnh}
		We consider the cone $\mathcal{C}_6^3$. The set $S_6^3$ has size $3\times (6-3)=9$. In fact, it is equal to $S_6$, and it looks like the following.
		\begin{equation*}
			\begin{matrix}
				(0,2) & (1,3) & (2,4) & (3,5) & (4,6)\\
				& (1,4) & (2,5) & (3,6) &\\
				& & (2,6) & &
			\end{matrix}
		\end{equation*}
		
		We want to show that $\dim \mathcal{C}_6^2=3\times (6-3-1) + 1 = 7$, so we need to find seven linearly independent diagrams in $\mathcal{C}_6^3$.
		
		We can do this by imposing an ordering $\prec_3$ on $S_6^3$, and finding $(i,d)$-initial graphs with respect to that ordering for seven values of $(i,d)$ in $S_6^3$.
		
		By corollaries \ref{cor-Kl} and \ref{cor-Cl}, Lemma \ref{beta-Cn} and Lemma \ref{beta-L}, we know the shapes of the following Betti diagrams. Beside each diagram $\beta$, we note down a value of $(i,d)$ such that $\beta_{i,d}$ is non-zero but for every diagram $\beta'$ listed above it, $\beta'_{i,d}$ is zero.
			\begin{center}
			\begin{tabular}{ l|l|l }
				$(i,d)$&Betti Diagram & Shape of Betti Diagram \\
				\hline
				$(2,4)$&$\beta^c(E_3+K_1)$ & $\begin{bmatrix}
					\beta_{0,2} & \beta_{1,3} & \beta_{2,4}
				\end{bmatrix}$\\
				&& \\
				$(3,5)$&$\beta^c(E_3+K_2)$ & $\begin{bmatrix}
					\beta_{0,2} & \beta_{1,3} & \beta_{2,4} & \beta_{3,5}
				\end{bmatrix}$\\
				&& \\
				$(4,6)$&$\beta^c(E_3+K_3)$ & $\begin{bmatrix}
					\beta_{0,2} & \beta_{1,3} & \beta_{2,4} & \beta_{3,5}& \beta_{4,6}
				\end{bmatrix}$\\
				&& \\
				$(2,5)$&$\beta^c(C_5)$ & $\begin{bmatrix}
					\beta_{0,2} & \beta_{1,3} &\\
					&&\beta_{2,5}
				\end{bmatrix}$\\
				&& \\
				$(1,4)$&$\beta^c(C_4+E_1)$ & $\begin{bmatrix}
					\beta_{0,2} & \beta_{1,3} & \beta_{2,4}\\
					&\beta_{1,4}&\beta_{2,5}
				\end{bmatrix}$\\
				&& \\
				$(3,6)$&$\beta((K_2+E_2)^c+L)$ & $\begin{bmatrix}
					\beta_{0,2} & \beta_{1,3} & \beta_{2,4}&\beta_{3,5}&\beta_{4,6}\\
					&\beta_{1,4}&\beta_{2,5}&\beta_{3,6}&
				\end{bmatrix}$\\
				&& \\
				$(2,6)$&$\beta(L+L+L)$ & $\begin{bmatrix}
					\beta_{0,2} &&\\
					&\beta_{1,4}&\\
					& & \beta_{2,6}
				\end{bmatrix}$\\
			\end{tabular}
		\end{center}
	
	If we define our ordering $\prec_3$ as $(0,2)\prec_3 (1,3)\prec_3 (2,4)\prec_3 (3,5) \prec_3 (4,6) \prec_3 (2,5) \prec_3 (1,4) \prec_3 (3,6) \prec_3 (2,6)$, then the diagrams we have found are $(i,d)$-initial with respect to $\prec_3$.
	
	This ordering is the same as the ordering $\prec$ given in Definition \ref{def: order-Cn}, except that from row 2 onwards, the order of the elements for which $i<3$ is reversed. We can generalise this construction to get an ordering $\prec_h$ on $S_n^h$ for arbitrary values of $h$.
	
	The graphs corresponding to these diagrams each have $6$ or fewer vertices, and we can show that each of them has height $3$ using Lemma \ref{ht}. Thus all of these diagrams lie in $\mathcal{C}_6^3$, as required.
	
	The final two diagrams in the table come from adding the graph $L$ to graphs of height $h-1$ which are $(i-1,d-2)$-initial with respect to $\prec_{h-1}$. Again, we generalise this process in what follows.\\
	\end{ex}
	
	With the above example in mind, we proceed to defining an ordering $\prec_h$ on $S_n^h$.
	
	\begin{defn}\label{def: order-Cnh}
		For any two pairs $(i,d)$ and $(i',d')$ in $S_n^h$ we write $(i,d)\prec_h (i',d')$ if	any one of the following conditions hold.
		\begin{enumerate}
			\item $d-i<d'-i'$.
			\item $d-i=d'-i'=2$ and $i<i'$.
			\item $d-i=d'-i'>2$, $h\leq i'$, and $i<i'$.
			\item $d-i=d'-i'>2$, $i,i'<h$ and $i>i'$.
		\end{enumerate}
	\end{defn}

	In other words, if $(i,d)$ and $(i',d')$ both lie in the same row, higher than 1, with $i$ and $i'$ less than $h$, then we write $(i,d)\prec_h(i',d')$ provided that $i>i'$. In all other cases, we write $(i,d)\prec_h (i',d')$ if $(i,d)\prec (i,d)$.
 	
 	We have an analogue of Lemma \ref{G+L initiality} for the ordering $\prec_h$.
 	\begin{lem}\label{G+L initiality Cnh}
 		Let $\prec_h$ be as in Definition \ref{def: order-Cnh}, let $(i,d)\in S_n$, and suppose $G$ is a graph of height $h-1$ on vertex set $[n-2]$, which is $(i-1,d-2)$-initial with respect to the ordering $\prec_{h-1}$. The graph $G+L$ has height $h$, and is $(i,d)$-initial with respect to the ordering $\prec_h$.
 	\end{lem}
 	\begin{proof}
 		Any vertex cover of size $h-1$ for $G$ can be extended to a vertex cover of size $h$ for $G+L$ by adding one of the two vertices in $L$. Moreover, this vertex cover for $G+L$ is of minimal size, so by Lemma \ref{ht} we conclude that $G+L$ has height $h$.
 		
 		To prove that $G+L$ is $(i,d)$-initial with respect to $\prec_h$, we need to consider a few separate cases.
 		
 		First we consider the case $i \geq h$. In this case, $(i,d)\prec_h (i',d')$ if and only if $(i,d)\prec (i',d')$. Similarly, because $i-1 \geq h-1$, we have $(i-1,d-2)\prec_{h-1} (i'-1,d'-2)$ if and only if $(i'-1,d'-2)\prec (i-1,d-2)$. So the result follows immediately from Lemma \ref{def: order-Cn}.
 		
 		Note also that $(i-1,d-2)$ is in a lower numbered row of $S_n^h$ than $(i,d)$, so $(i,d)$ cannot be in row 1. So we may assume that $(i,d)$ lies in a row greater than 1, and $i< h$. This means $i-1< h-1$.
 		
 		By Proposition \ref{beta-L}, we have $\beta_{i,d}(G+L) = \beta_{i,d}(G) + \beta_{i-1,d-2}(G)$. By the $(i-1,d-2)_{\prec_{h-1}}$-initiality of $G$, we have $\beta_{i-1,d-2}(G) \neq 0$, so $\beta_{i,d}(G+L)$ must be non-zero too.
 		
 		Now let $(i',d')\in S_n$ with $(i,d)\prec_h (i',d')$. Again, by Proposition \ref{beta-L} we have $\beta_{i',d'}(G+L) = \beta_{i',d'}(G) + \beta_{i'-1,d'-2}(G)$. It is easy to check that if $(i,d)\prec_{h}(i',d')$, then $(i-1,d-2)\prec_{h-1}(i'-1,d'-2)$, and we must have $(i'-1,d'-2)\prec_{h-1} (i',d')$ because they are in different rows. Thus, by the $(i-1,d-2)_{\prec_{h-1}}$-initiality of $G$, both the terms $\beta_{i',d'}(G)$ and $\beta_{i'-1,d'-2}(G)$ are zero, and hence $\beta_{i',d'}(G+L)$ is zero too.
 		
 		This shows that $\beta(G+L)$ is $(i,d)$-initial with respect to the ordering $\prec_h$ as required.
 	\end{proof}
 	
 	Finally, we can move on to proving that the formula in Theorem \ref{dimCnh} is a lower bound for $\dim\mathcal{C}_n^h$.
 	\begin{prop}\label{dimCnh lower bound}
 		We have $\dim \mathcal{C}_n^h \geq h(n-h-1)+1$.
 	\end{prop}
	\begin{proof}
 		We want to show that there exists a set of graphs $\{G^h_{i,d}|(i,d)\in S_n^h-\{(0,2),...,(h-2,h)\}$ such that for each $(i,d)$ in $S_n$, $G^h_{i,d}$ is a graph of height $h$ which is $(i,d)$-initial with respect to $\prec_h$. We also show that these graphs have $n$ or fewer vertices, which proves the result, because it ensures that the Betti diagrams of these graphs all lie inside $\mathcal{C}_n^h$ by Lemma \ref{Cmh in Cnh}.
 		
 		Specifically, we need to ensure that every graph $G^h_{i,d}$ we find has no more than $\max\{d,h+d-i-1\}$ vertices. This is sufficient because for every $(i,d)\in S_n^h$, we have $d\leq n$ and also $d-i\leq \min \{h,n-h\}+1\leq n-h +1$, so $h+d-i-1\leq h+(n-h+1)-1=n$.
 		
 		As in example \ref{ex-Cnh}, we define a lot of these graphs in terms of their complements $H^h_{i,d}=(G^h_{i,d})^c$
 		
 		We start by finding $(i,d)$-initial graphs for all the values of $(i,d)$ in row 1 of $S_n^h-\{(0,2),...,(h-2,h)\}$ (i.e. those for which $d-i=2$). So let $(i,d)$ be any such value. We have assumed that $(i,d)\neq (0,2),...,(h-2,h)$ so we know that $d>h$.
 		
 		This allows us to define $H^h_{i,d}=E_h+K_{d-h}$. By Corollary \ref{cor-Kl}, the complement of this graph is $(i,d)$-initial with respect to $\prec$ which is the same as being $(i,d)$-initial with respect to $\prec_h$ for the values of $(i,d)$ in row 1. We can obtain a complete induced subgraph of $H^h_{i,d}$ by removing the $h$ vertices in $E_h$, which means that its complement has height $h$ by Lemma \ref{ht}. Also it has $d$ vertices.
 		
 		This completes the proof for the case $h=1$ and $h=n-1$, because $S_n^1=S_n^{n-1}$ only has one row. We can now proceed by induction on $h>1$, using the completed $h=1$ case as our base case. So we assume that $2\leq h\leq n-2$, and that we have found an appropriate set of graphs for the cone $\mathcal{C}_n^{h-1}$.
 		
 		Let $(i,d)$ be in row 2 of $S_n^h$ (i.e. $d-i=3$) and suppose $i<h$. We can define $H_{i,d}^h=C_d+E_{h+2-d}$. By Corollary \ref{cor-Kl}, the complement of this graph is $(i,d)$-initial with respect to $\prec_h$. Also, we can obtain a complete induced subgraph of $H^h_{i,d}$ by removing $h$ of the $h+2$ vertices and leaving any two adjacent vertices in $C_d$. By Lemma \ref{ht}, this means that its complement has height $h$. Also it has $(h+2)=(h+d-i-1)$ vertices.
 		
 		For every other value of $(i,d)$ in $S_n^h$, we can define $G_{i,d}^h=G_{i-1,d-2}^{h-1} +L$. By Lemma \ref{ht}, these values of $G_{i,d}$ have height $h$ and by Proposition \ref{G+L initiality Cnh}, they are $(i,d)$-initial with respect to $\prec_h$, as required.
 		
 		It only remains to show that the number of vertices of these graphs is bounded above by at least one of $d$ or $(h+d-i-1)$. If the graph $G^{h-1}_{i-1,d-2}$ has at most $d-2$ vertices, then $G^h_{i,d}$ has at most $d$ vertices. On the other hand, if $G^{h-1}_{i,d}$ has at most $(h-1)+(d-2)-(i-1)-1=h+d-i-3$ vertices, then $G^h_{i,d}$ has at most $(h+d-i-1)$ vertices. This completes the proof.\\
 	\end{proof}

 	This result shows that, up to linear combinations, the Herzog-K\"{u}hl equations are the \textit{only} non-trivial linear dependency relations that hold for every diagram in $\mathcal{C}_n^h$. The term `\textit{non-trivial}' here refers to relations that only feature variables $\beta_{i,d}$ for which $(i,d)\in S_n^h$. By Proposition \ref{Snh}, any Betti number whose index lies outside $S_n^h$ must be zero, so any linear relation in those Betti numbers is satisfied trivially.
 	
 	This is clear because any non-trivial relation $r(\beta)=0$ defines a hyperplane in $(W_n^h)'$ containing $\mathcal{C}_n^h$, and this hyperplane must contain $W_n^h$, which means $r(\beta)=0$ must be expressible as a linear combination of the relations $HK_1(\beta)=0,...,HK_{h-1}(\beta)=0$. So in this sense, the Herzog-K\"{u}hl equations are in fact best possible for the diagrams in $\mathcal{C}_n^h$.
 		
 	Similarly, the fact that the space $W_n$ is the minimal subspace of $V_n$ containing $\mathcal{C}_n$ shows that there are \textit{no} non-trivial linear dependency relations that hold for every diagram in $\mathcal{C}_n$.
 	
	\pagebreak

	
\end{document}